\documentclass[11pt]{article}
\usepackage[utf8]{inputenc}
\usepackage{amsmath, amssymb, amsthm}

\usepackage{comment}
\usepackage{tikz} 
\usepackage{authblk}

\usetikzlibrary{arrows,automata,positioning} 

\newtheorem{lemma}{Lemma}
\newtheorem{corollary}{Corollary}
\newtheorem{theorem}{Theorem}

\title{A Note on the Maximum Number of Minimal Connected Dominating Sets in a Graph}
\author{Faisal N. Abu-Khzam}
\affil[]{Department of Computer Science and Mathematics\\ Lebanese American University\\ Beirut, Lebanon}
\date{}
\begin{document}

\maketitle

\thispagestyle{empty}

\begin{abstract}

We prove constructively that the maximum possible number of minimal connected dominating sets in a connected undirected graph of order $n$ is in $\Omega(1.489^n)$. This improves the previously known lower bound of $\Omega(1.4422^n)$ and reduces the gap between lower and upper bounds for input-sensitive enumeration of minimal connected dominating sets in general graphs as well as some special graph classes.
\end{abstract}

\section{Introduction}

A connected dominating set in a graph $G=(V,E)$ is a set of vertices whose closed neighborhood is $V$ that induces a connected subgraph. A connected dominating set is inclusion minimal if it does not contain another connected dominating set as a proper subset. 

Enumerating all minimal connected dominating sets in a given graph can be trivially performed in $O(2^n)$. Whether a better enumeration algorithm exists was one of the most important open problems posed in the first workshop on enumeration (Lorentz Center, Netherlands, 2015) \cite{lorentz}.
The problem has been subsequently addressed in \cite{belowAllsubsets} where an algorithm that runs in $O((2-10^{-50})^n)$ was presented.
This slightly improves the upper bound on the number of minimal connected dominating sets in a (general) graph. 

On the other hand, the maximum number of minimal connected dominating sets in a graph was shown to be in $\Omega(3^{\frac{n}{3}})$ \cite{GHK16}. This lower bound is obviously very low compared to the upper bound and to the running time of the current asymptotically-fastest exact algorithm, which is in $O(1.862^n)$ \cite{abu2011}. 
The gap between upper and lower bounds is narrower when it comes to special graph classes. On chordal graphs, for example, the upper bound has been recently improved to $O(1.4736^n)$ \cite{Petr2020}. Other improved lower/upper bounds have been obtained for AT-free, strongly chordal, distance-hereditary graphs, and cographs in \cite{GHK16}. Further improved bounds for split graphs and cobipartite graphs have been obtained in \cite{skjorten2017faster}.

In this note we report an improved lower bound on the maximum number of minimal connected dominating sets in a graph. This is related to the enumeration of all the minimal connected dominating sets since it also gives a lower bound on the asymptotic performance of any input-sensitive enumeration algorithm.

\section{Graphs with Large Minimal Connected Dominating Sets}

Given arbitrary positive integers $k,t$, we construct a graph $G_t^k$ of order $n = k(2t+1) + 1$ as follows. 

The main building blocks of $G_t^k$ consist of $k$ copies of a base-graph $G_{t}$, 
of order $2t-1$. The vertex set of $G_{t}$ consists of three layers. The first layer is a set  $X=\{x_{1}\ldots x_{t}\}$ that induces a clique. The second is an independent set $Y=\{y_{1},\ldots y_{t}\}$, while the third layer consists of a singleton $\{z\}$.
Each vertex $x_{i}\in X$ has exactly $t-1$ neighbors in $Y$: $N(x_{i}) = \{y_{j}\in Y: i\neq j\}$. In other words, the base-graph $G_{t}$ has a maximum anti-matching\footnote{An anti-matching in $G$ is a collection of disjoint non-adjacent pairs of its vertices. } $\{\{x_{j},y_{j}\}: 1\leq j\leq t\}$. 
In fact, $X\cup Y$ induces a copy of $K_{t,t}$ minus a perfect matching. Finally the vertex $z$ is adjacent to all the $t$ vertices in $Y$. 
Figure \ref{basic-block} below shows the graph $G_t$ for $t=4$.

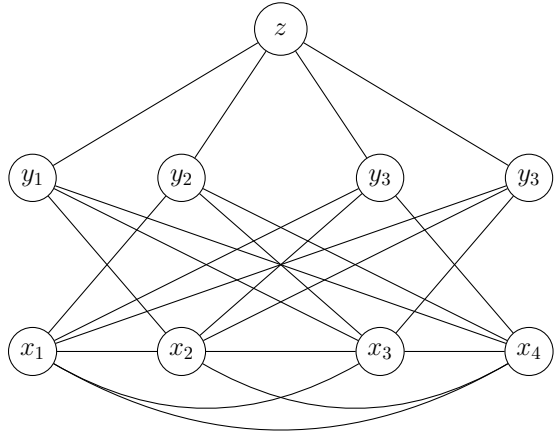
\begin{figure}[htb!]

\centering
\begin{tikzpicture}[->,
scale=0.33, every node/.style={anchor=center, scale=0.7},node distance=1cm,main node/.style={circle,fill=white!20,draw,font=\sffamily\Large\bfseries}]

\node [main node] (1) at (-2,2) {\(~z~\)};
\node[main node] (2) at (-12,-4) {\(y_1\)};
\node [main node] (3) at (-6,-4) {\(y_2\)};
\node [main node] (4) at (2,-4) {\(y_3\)};
\node [main node] (5) at (8,-4) {\(y_3\)};
\node[main node] (6) at (-12,-11) {\(x_1\)};
\node [main node] (7) at (-6,-11) {\(x_2\)};
\node [main node] (8) at (2,-11) {\(x_3\)};
\node [main node] (9) at (8,-11) {\(x_4\)};

\begin{scope}[-]

\draw [thin] (1) -- (2);
\draw [thin] (1) -- (3);
\draw [thin] (1) -- (4);
\draw [thin] (1) -- (5);
\draw [thin] (2) -- (7);
\draw [thin] (2) -- (8);
\draw [thin] (2) -- (9);
\draw [thin] (3) -- (6);
\draw [thin] (3) -- (8);
\draw [thin] (3) -- (9);
\draw [thin] (4) -- (6);
\draw [thin] (4) -- (7);
\draw [thin] (4) -- (9);
\draw [thin] (5) -- (6);
\draw [thin] (5) -- (7);
\draw [thin] (5) -- (8);
\draw [thin] (6) -- (7);
\draw [thin] (6) to[bend right] (8);
\draw [thin] (6) to[bend right] (9);
\draw [thin] (7) -- (8);
\draw [thin] (8) -- (9);
\draw [thin] (7) to[bend right] (9);

\end{scope}

\end{tikzpicture}

\caption{The graph $G_4$ }
\label{basic-block}
\end{figure}

\begin{lemma}
\label{Gi}
For each $t>0$, the graph $G_t$ has exactly $\frac{t^3+t^2}{2}-t$ minimal connected dominating sets that have non-empty intersection with the set $X$.
\end{lemma}

\begin{proof}
The set $X$ cannot have more than two vertices in common with any minimal connected dominating set, since any two elements of $X$ dominate $X\cup Y$. Any minimal connected dominating set that contains exactly one vertex $x_i$ of $X$ must contain the vertex $z$, to dominate $y_i$, and one of the $t-1$ neighbors of $x_i$ (to be connected). There are $t(t-1)$ sets of this type.
Moreover, each pair of elements of $X$ dominates $Y$. So a minimal connected dominating set can be formed by (any) two elements of $X$ and any of the elements of $Y$ (to dominate $z$). There are $t\frac{t(t-1)}{2}$ such sets. 
\end{proof}

The hub-vertex $s$ in $G_t^k$ must be in any connected dominating set, being a cut-vertex. Therefore, there is no need for the set $X$ in $G_t$ to induce a clique (in $G_t^k$), being always dominated by $s$. In other words, the counting used in the above proof still holds if each copy of $G_t$ is replaced by $G_{t}-E(X)$ in $G_t^k$. Here $E(X)$ denotes the set of edges connecting pairs of vertices in $X$. The below figure shows $G_3^3$ without the edges between pairs of element of $X$ in each copy of $G_3$. 

\vspace{5pt}

\begin{figure}[htb!]

\centering
\begin{tikzpicture}[->,
scale=0.29, every node/.style={anchor=center, scale=0.6},node distance=1cm,main node/.style={circle,fill=white!20,draw,font=\sffamily\Large\bfseries}]

\node [main node] (1) at (-17,0) {\(~z_1~\)};
\node[main node] (2) at (-22,-4) {\(y_{11}\)};
\node [main node] (3) at (-17,-4) {\(y_{12}\)};
\node [main node] (4) at (-12,-4) {\(y_{13}\)};
\node[main node] (5) at (-22,-8) {\(x_{11}\)};
\node [main node] (6) at (-17,-8) {\(x_{12}\)};
\node [main node] (7) at (-12,-8) {\(x_{13}\)};

\node [main node] (8) at (-2,0) {\(~z_2~\)};
\node[main node] (9) at (-7,-4) {\(y_{21}\)};
\node [main node] (10) at (-2,-4) {\(y_{22}\)};
\node [main node] (11) at (3,-4) {\(y_{23}\)};
\node[main node] (12) at (-7,-8) {\(x_{21}\)};
\node [main node] (13) at (-2,-8) {\(x_{22}\)};
\node [main node] (14) at (3,-8) {\(x_{23}\)};

\node [main node] (15) at (13,0) {\(~z_3~\)};
\node[main node] (16) at (8,-4) {\(y_{31}\)};
\node [main node] (17) at (13,-4) {\(y_{32}\)};
\node [main node] (18) at (18,-4) {\(y_{33}\)};
\node[main node] (19) at (8,-8) {\(x_{31}\)};
\node [main node] (20) at (13,-8) {\(x_{32}\)};
\node [main node] (21) at (18,-8) {\(x_{33}\)};

\node [main node] (22) at (-2,-18) {\(~s~\)};


\begin{scope}[-]

\draw [thin] (8) -- (9);
\draw [thin] (8) -- (10);
\draw [thin] (8) -- (11);
\draw [thin] (9) -- (13);
\draw [thin] (9) -- (14);
\draw [thin] (10) -- (12);
\draw [thin] (10) -- (14);
\draw [thin] (11) -- (12);
\draw [thin] (11) -- (13);

\draw [thin] (22) -- (5);
\draw [thin] (22) -- (6);
\draw [thin] (22) -- (7);
\draw [thin] (22) -- (12);
\draw [thin] (22) -- (13);
\draw [thin] (22) -- (14);
\draw [thin] (22) -- (19);
\draw [thin] (22) -- (20);
\draw [thin] (22) -- (21);

\draw [thin] (1) -- (2);
\draw [thin] (1) -- (3);
\draw [thin] (1) -- (4);
\draw [thin] (2) -- (6);
\draw [thin] (2) -- (7);
\draw [thin] (3) -- (5);
\draw [thin] (3) -- (7);
\draw [thin] (4) -- (5);
\draw [thin] (4) -- (6);

\draw [thin] (15) -- (16);
\draw [thin] (15) -- (17);
\draw [thin] (15) -- (18);
\draw [thin] (16) -- (20);
\draw [thin] (16) -- (21);
\draw [thin] (17) -- (19);
\draw [thin] (17) -- (21);
\draw [thin] (18) -- (19);
\draw [thin] (18) -- (20);

\end{scope}

\end{tikzpicture}

\caption{The graph $G_3^3$ }
\label{G3-3}
\end{figure}
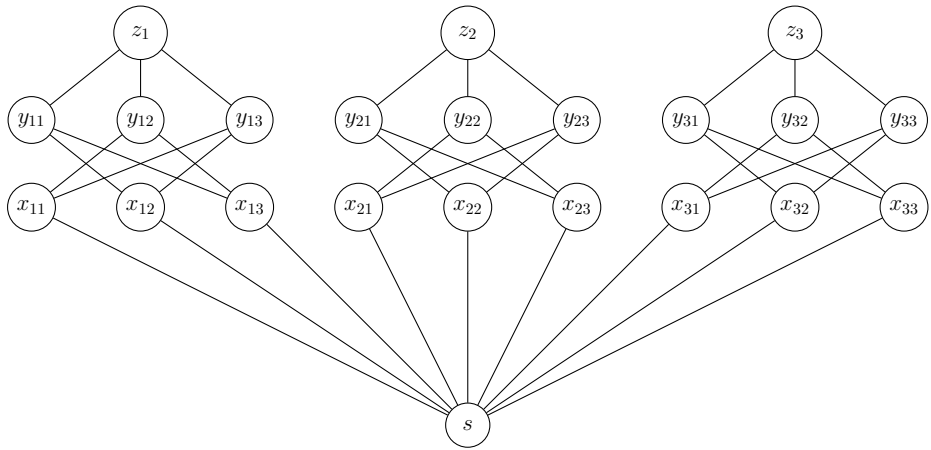

\begin{theorem}
\label{lowerbound}
The maximum number of minimal connected dominating sets in a simple undirected connected graph of order $n$ is in $\Omega(1.489^n)$.
\end{theorem}

\begin{proof}
By Lemma \ref{Gi}, each copy of the graph $G_{t}$ has $\frac{t^3+t^2}{2}-t$ minimal connected dominating sets. There are $k$ such graphs in $G_t^k$, in addition to the vertex $s$ that connects them all. 
Every minimal connected dominating set must contain $s$ and at least one element from $N(s)$ in each $G_{t}$.
Therefore, the total number of minimal connected dominating sets in $G_t^k$ is $(\frac{t^3+t^2}{2}-t)^k = (\frac{t^3+t^2}{2}-t)^{\frac{n-1}{2t+1}}$. The claimed lower bound is achieved when $t=4$, which gives a total of $36^{\frac{n-1}{9}} \in \Omega(1.489^n)$.
\end{proof}


We note that $G_t^k$ is a $t$-degenerate  graph that is also bipartite (since the set $X$ in each copy of $G_{t}$ can be an independent set). Furthermore, we observe that $G_3^k$ is planar. To see this, simply re-order the elements of $Y$ in each copy of $G_3$ as shown in Figure \ref{planar} below. 



\vspace{-20pt}

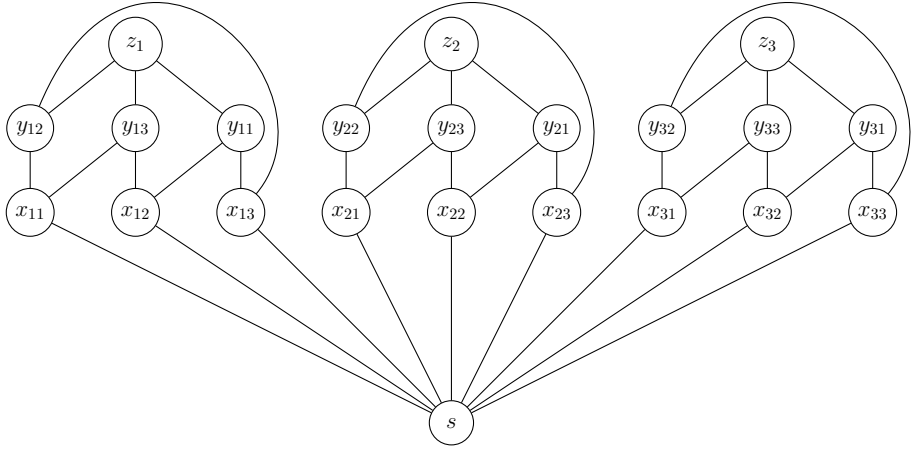
\begin{figure}[htb!]

\centering
\begin{tikzpicture}[->,
scale=0.28, every node/.style={anchor=center, scale=0.6},node distance=1cm,main node/.style={circle,fill=white!20,draw,font=\sffamily\Large\bfseries}]

\node [main node] (1) at (-17,0) {\(~z_1~\)};
\node[main node] (2) at (-22,-4) {\(y_{12}\)};
\node [main node] (3) at (-17,-4) {\(y_{13}\)};
\node [main node] (4) at (-12,-4) {\(y_{11}\)};
\node[main node] (5) at (-22,-8) {\(x_{11}\)};
\node [main node] (6) at (-17,-8) {\(x_{12}\)};
\node [main node] (7) at (-12,-8) {\(x_{13}\)};

\node [main node] (8) at (-2,0) {\(~z_2~\)};
\node[main node] (9) at (-7,-4) {\(y_{22}\)};
\node [main node] (10) at (-2,-4) {\(y_{23}\)};
\node [main node] (11) at (3,-4) {\(y_{21}\)};
\node[main node] (12) at (-7,-8) {\(x_{21}\)};
\node [main node] (13) at (-2,-8) {\(x_{22}\)};
\node [main node] (14) at (3,-8) {\(x_{23}\)};

\node [main node] (15) at (13,0) {\(~z_3~\)};
\node[main node] (16) at (8,-4) {\(y_{32}\)};
\node [main node] (17) at (13,-4) {\(y_{33}\)};
\node [main node] (18) at (18,-4) {\(y_{31}\)};
\node[main node] (19) at (8,-8) {\(x_{31}\)};
\node [main node] (20) at (13,-8) {\(x_{32}\)};
\node [main node] (21) at (18,-8) {\(x_{33}\)};

\node [main node] (22) at (-2,-18) {\(~s~\)};

\node[] (50) at (-17,2) {};


\begin{scope}[-]

\draw [thin] (8) -- (9);
\draw [thin] (8) -- (10);
\draw [thin] (8) -- (11);
\draw [thin] (12) -- (9);
\draw [thin] (12) -- (10);
\draw [thin] (13) -- (10);
\draw [thin] (13) -- (11);
\draw [thin] (14) -- (11);



\draw [thin] (22) -- (5);
\draw [thin] (22) -- (6);
\draw [thin] (22) -- (7);
\draw [thin] (22) -- (12);
\draw [thin] (22) -- (13);
\draw [thin] (22) -- (14);
\draw [thin] (22) -- (19);
\draw [thin] (22) -- (20);
\draw [thin] (22) -- (21);

\draw [thin] (1) -- (2);
\draw [thin] (1) -- (3);
\draw [thin] (1) -- (4);
\draw [thin] (5) -- (2);
\draw [thin] (5) -- (3);
\draw [thin] (6) -- (3);
\draw [thin] (6) -- (4);
\draw [thin] (7) -- (4);
\draw (2) .. controls (-17,8) and (-7,-2) .. (7);

\draw (9) .. controls (-2,8) and (8,-2) .. (14);

\draw (16) .. controls (13,8) and (23,-2) .. (21);

\draw [thin] (15) -- (16);
\draw [thin] (15) -- (17);
\draw [thin] (15) -- (18);
\draw [thin] (19) -- (16);
\draw [thin] (19) -- (17);
\draw [thin] (20) -- (17);
\draw [thin] (20) -- (18);
\draw [thin] (21) -- (18);


\end{scope}

\end{tikzpicture}

\caption{A plane drawing of $G_3^3$ }
\label{planar}
\end{figure}

Therefore, we can obtain an improved lower bound for 3-degenerate, bipartite and planar graphs. We conclude with the following corollary.

\begin{corollary}
The maximum number of minimal connected dominating sets in a 3-degenerate bipartite planar graph of order $n$ is in $\Omega(1.472^n)$.
\end{corollary}

\section{Conclusion}

The method we adopted for constructing asymptotic worst-case examples for  enumerating minimal connected dominating sets consists of combining copies of a certain base-graph 
having a particular subset of vertices that must contain elements of any minimal connected dominating set, being linked to a main hub-vertex. For example, the graph $G_{4}$ has 36 minimal connected dominating sets, each of which must contain elements of the set $X$, which in turn is linked to the hub-vertex $s$ in $G_4^k$.

The main question at this stage is: can we do better? 
We believe it is very difficult to find a base-graph of order eight or less that can be used to achieve a higher lower-bound since it would have to have at least 25 minimal connected dominating sets. Moreover, any better example that contains more than 9 vertices must have a  much larger number of minimal connected dominating sets. For example, to achieve a better lower bound with a base-graph of order 10 (or 11), such a graph must have at least 54 (respectively 80) minimal connected dominating sets. It would be challenging to obtain such a construction, which is hereby posed as an open problem.



\end{document}